\documentclass[11 pt]{amsart}

\usepackage[latin1]{inputenc}
\usepackage{amsmath,amsthm,amssymb,amscd,color, xcolor,mathtools,url}
\usepackage{bbm}

\newtheorem{thm}{Theorem}[section]
 
 \newtheorem{lem}[thm]{Lemma}

 \theoremstyle{definition}
 
 \theoremstyle{remark}

 \numberwithin{equation}{section}

\def\be#1 {\begin{equation} \label{#1}}
\def\ee{\end{equation}}

\def\sqw{\hbox{\rlap{\leavevmode\raise.3ex\hbox{$\sqcap$}}$%
\sqcup$}}
\def\findem{\ifmmode\sqw\else{\ifhmode\unskip\fi\nobreak\hfil
\penalty50\hskip1em\null\nobreak\hfil\sqw
\parfillskip=0pt\finalhyphendemerits=0\endgraf}\fi}

\newcommand{\R}{\mathbb R}

\newcommand{\N}{\mathbb N}
\newcommand{\Z}{\mathbb Z}

\newcommand{\la}{\lambda}
\newcommand{\de}{\delta}

\title{Bounds for spectral projectors on the Euclidean cylinder}

\begin{document}

\begin{abstract}
	We prove essentially optimal bounds for norms of spectral projectors on thin spherical shells for the Laplacian on the cylinder \((\mathbb{R} / \mathbb{Z}) \times \mathbb{R}\). In contrast to previous investigations into spectral projectors on tori, having one unbounded dimension available permits a compact self-contained proof.
\end{abstract}

\author[P. Germain and S. L. Rydin Myerson]{Pierre Germain and Simon L. Rydin Myerson}

\maketitle


\section{Introduction}

\subsection{Spectral projectors on general manifolds and tori} Given a Riemannian manifold with Laplace-Beltrami operator $\Delta$, consider the spectral projector $P_{\lambda,\delta}$ on (perhaps generalized) eigenfunctions with eigenvalues within $O(\delta)$ of $\lambda$. It is defined through functional calculus by the formula
$$
P_{\lambda,\delta} = P_{\lambda,\delta}^{\chi} = \chi \left( \frac{\sqrt{-\Delta} - \lambda}{\delta} \right),
$$
where $\chi$ is a cutoff function, which is irrelevant for our purposes.

An interesting question is to determine the operator norm from $L^2$ to $L^p$, with $p>2$, of this operator. A theorem of Sogge~\cite{Sogge} gives an optimal answer for any Riemannian manifold if $\delta =1$
$$
\| P_{\lambda,1} \|_{L^2 \to L^p} \lesssim \lambda^{\frac{d-1}{2} - \frac{d}{p}} + \lambda^{\frac{d-1}{2} \left( \frac{1}{2} - \frac{1}{p} \right)}.
$$
While this completely answers the question if $\delta>1$, the case $\delta<1$ is still widely open. Understanding the case $\delta<1$ requires a global analysis on the Riemannian manifold, which makes it very delicate.

In the case of the rational torus $\mathbb{R}^d / \mathbb{Z}^d$, $L^p$ bounds on eigenfunctions attracted a lot of attention; this corresponds to the choice $\delta = 1/\lambda$. The best result in this direction is due to Bourgain and Demeter~\cite{BourgainDemeter3}. More recently, the authors of the present paper~\cite{GM} considered the problem for general values of $\lambda$ and $\delta$, conjectured the bound for general tori
$$
\| P_{\lambda,\delta} \|_{L^{2} \to L^p} \lesssim  \lambda^{\frac{d-1}{2} - \frac{d}{p}} \delta^{1/2} + (\lambda \delta)^{\frac{(d-1)}{2} \left( \frac{1}{2} - \frac{1}{p} \right)}  \qquad \mbox{for $\delta > 1/ \lambda$},
$$
and were able to establish this bound for a range of the parameters $\delta,\lambda,p$.

A full proof of this conjecture seems very challenging in every dimension $d$. Restricting to the case $d=2$, consider the case $ (\mathbb{R} / \mathbb{Z}) \times \mathbb{R} = \mathbb{T} \times \mathbb{R}$ instead of $\mathbb{T}^2$. The conjecture remains identical, but a short proof, relying on $\ell^2$ decoupling, can be provided; this is the main observation of the present paper. Generalizations to higher dimensions are certainly possible.

\subsection{The Euclidean cylinder}

On $\mathbb{T} \times \mathbb{R} = (\mathbb{R} / \mathbb{Z}) \times \mathbb{R}$, we choose coordinates $(x,y)$, with $x \in [0,1]$ and $y \in \mathbb{R}$. The Laplacian operator is  given by
$$
\Delta = \partial_x^2 + \partial_y^2.
$$

A function $f$ on $\mathbb{T} \times \mathbb{R}$ can be expanded through Fourier series in $x$ and Fourier transform in $y$:
$$
f(x) = \sum_{k \in \mathbb{Z}} \int_{\R} \widehat{f}(k,\eta) e^{2\pi i (kx+\eta y)} \,d\eta.
$$
The spectral projector can then be expressed as
$$
P_{\lambda,\delta}  f(x) = \sum_{k \in \mathbb{Z}} \int_{\R} \chi \left( \frac{\sqrt{k^2 + \eta^2} - \lambda}{\delta} \right) \widehat{f}(k,\eta) e^{2\pi i (kx+\eta y)} \,d\eta.
$$

\begin{thm} \label{mainthm} If $\lambda > 1$ and $\delta < 1$,
$$
\| P_{\lambda,\delta} \|_{L^2 \to L^p} \lesssim_\epsilon \lambda^\epsilon \delta^{-\epsilon} \left[ \lambda^{\frac{1}{2} - \frac{2}{p}} \delta^{\frac{1}{2}} + \left( \lambda \delta \right)^{\frac{1}{4} - \frac{1}{2p}} \right]
$$
Furthermore, this estimate is optimal, up to the subpolynomial factor $\lambda^\epsilon \delta^{-\epsilon}$ and the multiplicative constant.
\end{thm}

\subsection{Strichartz estimates} It is interesting to draw a parallel with Strichartz estimates in dimension 2 for the Schr\"odinger equation, in which case the critical exponent equals 4. It was proved in the foundational paper of Bourgain~\cite{JBourgain} that
$$
\| e^{it\Delta} f \|_{L^4([0,1] \times \mathbb{T}^2)} \lesssim_s \| f \|_{H^s( \mathbb{T}^2)} \qquad \mbox{for $s>0$}.
$$
Takaoka and Tzvetkov~\cite{TT} proved that the above inequality fails for $s=0$, but that, on $\mathbb{T} \times \mathbb{R}$,
$$
\| e^{it\Delta} f \|_{L^4([0,1] \times \mathbb{T} \times \mathbb{R})} \lesssim \| f \|_{L^2( \mathbb{T}\times \mathbb{R})}.
$$
Finally, Barron, Christ and Pausader~\cite{BCP} determined the correct global (in time) estimate, for which a further summation index is needed. These examples suggest that optimal estimates might differ by subpolynomial factors between $\mathbb{T}^2$ and $\mathbb{T} \times \mathbb{R}$.

\subsection*{Acknowledgements} PG was supported by the Simons collaborative grant on weak turbulence. SLRM was supported by a Leverhulme Early Career Fellowship.

\section{Proof of the main theorem}
\begin{proof}
By Plancherel's theorem, it suffices to prove
\begin{equation}
\label{estimatef}
\| f \|_{L^p} \lesssim_\epsilon \lambda^\epsilon \delta^{-\epsilon} \left[ \lambda^{\frac{1}{2} - \frac{2}{p}} \delta^{\frac{1}{2}} + \left( \lambda \delta \right)^{\frac{1}{4} - \frac{1}{2p}} \right] \| f \|_{L^2}
\end{equation}
for $f$ a function whose Fourier transform is localized in the corona $\mathcal{C}_{\lambda,\delta}$ of radius $\lambda$ and with {thickness} $\delta/\lambda$:
$$
\mathcal{C}_{\lambda,\delta} = \{ (k,\eta) \; \mbox{such that} \: \lambda - \delta < \sqrt{k^2 + \eta^2} < \lambda + \delta
{\}}
.
$$
By symmetry, one can furthermore assume that $\widehat{f}(k,\eta)$ is localized in the first quadrant $k,\eta \geq 0$.

The function $f$ can be split into two pieces, which will correspond to the two terms on the right-hand side of~\eqref{estimatef}. 
$$
f(x) = \left[ \sum_{|k-\lambda| \leq  \frac{1}{\delta}} +  \sum_{|k-\lambda| > \frac{1}{\delta}} \right] \int_{\R}  \widehat{f}(k,\eta) e^{2\pi i (kx+\eta y)} \,d\eta = f_1(x) + f_2(x).
$$

\bigskip

\noindent \underline{The case $|k-\lambda| \leq \frac{1}{\delta}$} The Fourier support of $f_1$ is made up of a collection of segments. We will see in Lemma~\ref{Fouriersupport} below that the added length of these segments can be bounded by
$$
| \operatorname{Supp} \widehat{f_1} | \lesssim \sqrt{\lambda \delta}.
$$
Therefore, by the Cauchy-Schwartz inequality,
$$
\| f_1 \|_{L^\infty}
{\leq 
| \operatorname{Supp} \widehat{f_1} |^{1/2} \cdot
\| \widehat{f_1} \|_{L^2}^{1/2}}
\lesssim (\lambda \delta)^{1/4} \| f_1 \|_{L^2}.
$$
Interpolating with $L^2$, this gives
$$
\| f_1 \|_{{L^p}} \lesssim (\lambda \delta)^{\frac{1}{4} - \frac{1}{2p}} \| f \|_{L^2}.
$$

\bigskip

\noindent \underline{The case $|k-\lambda| > \frac{1}{\delta}$} We start by choosing a function $\phi \in \mathcal{S}$ which is $>1/2$ on $[-1,1]$, and has Fourier support in $[-1,1]$. 
We use periodicity in the $x$ variable to expand the range of \(x\) from \(x\in [0,1]\) to \(x< \delta^{-1}\), so that
$$
\left\| f_2 \right\|_{L^p(\mathbb{T} \times \mathbb{R})}
\lesssim \delta^{1/p} \left\| \phi \left( \delta x \right) \sum_{|k-\lambda| > \frac{1}{\delta}}  \int_{\R}  \widehat{f}(k,\eta) e^{2\pi i (kx+\eta y)} \,d\eta \right\|_{L^p (\mathbb{R}^2)}.
$$
We now change variables as follows: $X = \lambda x$, $Y = \lambda y$, $K = k / \lambda$, $H = \eta / \lambda$, 
$$
f_2(x,y) = \lambda F(X,Y), \qquad F(X,Y) = \phi \left( \frac{\delta X}{\lambda} \right) \sum_{\substack{K \in \mathbb{Z}/\lambda \\ |K - 1| > \frac{1}{\delta \lambda}}} \int_{\R}  \widehat{f}(\lambda K, \lambda H) e^{2\pi i (KX + HY)} \,dH 
$$
to obtain
$$
\left\| f_2 \right\|_{L^p(\mathbb{T} \times \mathbb{R})} \lesssim  \delta^{1/p}  \lambda^{1-\frac{2}{p}} \left\| F \right\|_{L^p (\mathbb{R}^2)}.
$$
The effect of this change of variables is that the function of $(X,Y)$ whose $L^p$ norm we want to evaluate has Fourier transform supported in the corona  $\mathcal{C}_{1,3\delta/\lambda}$ of radius $1$ and width $\delta/\lambda$, and also in the first quadrant $X,Y \geq 0$. This enables us to apply the $\ell^2$ decoupling theorem of Bourgain and Demeter~\cite{BourgainDemeter3}: for a smooth partition of unity $(\chi_\theta)$ corresponding to a suitable almost disjoint covering of $\mathcal{C}_{1,3\delta/\lambda}$ by caps $(\theta)$ of size $\sim \frac{\delta}{\lambda} \times \sqrt{\frac{\delta}{\lambda}}$,
\begin{align*}
& \left\| f_2 \right\|_{L^p(\mathbb{T} \times \mathbb{R})}  \lesssim
\delta^{\frac{1}{p}}  \lambda^{1-\frac{2}{p}} (\delta/ \lambda)^{-\frac{1}{4}+\frac{3}{2p}-\epsilon}
\left( \sum_{\theta} \left\| \chi_\theta(D) F \right\|_{L^p (\mathbb{R}^2)}^2 \right)^{1/2},
\end{align*}
where $\chi_\theta(D)$ is the Fourier multiplier with symbol $\chi_\theta$.

We now apply the inequality $\| g \|_{L^p(\mathbb{R}^d)} \lesssim \| g \|_{L^2} | \operatorname{Supp} \widehat{g} |^{\frac{1}{2} - \frac{1}{p}}$ (if $p \geq 2$), which follows by applying in turn  the Hausdorff-Young and H\"older inequalities, and then the Plancherel theorem.
Since by Lemma~\ref{Fouriersupport} below
$$
| \operatorname{Supp} \widehat{\chi_\theta(D) F} |
= | \operatorname{Supp} \chi_\theta\widehat{F} |
\lesssim \delta^{5/2} \lambda^{-3/2},
$$
it follows that
$$
 \left\| f_2 \right\|_{L^p(\mathbb{T} \times \mathbb{R})}  \lesssim
\delta^{\frac{1}{p}}  \lambda^{1-\frac{2}{p}} (\delta/ \lambda)^{-\frac{1}{4}+\frac{3}{2p}-\epsilon} ( \delta^{5/2} \lambda^{-3/2})^{\frac{1}{2} - \frac{1}{p}} \left( \sum_{\theta} \left\| \chi_\theta(D) F \right\|_{L^2 (\mathbb{R}^2)}^2 \right)^{1/2}.
$$
By almost orthogonality, this becomes
$$
 \left\| f_2 \right\|_{L^p(\mathbb{T} \times \mathbb{R})}  \lesssim
\delta^{\frac{1}{p}}  \lambda^{1-\frac{2}{p}} (\delta/ \lambda)^{-\frac{1}{4}+\frac{3}{2p}-\epsilon} ( \delta^{5/2} \lambda^{-3/2})^{\frac{1}{2} - \frac{1}{p}} \left\| F \right\|_{L^2(\mathbb{R}^2)}.
$$
Finally, undoing the change of variables and using once again periodicity in the $x$ variable gives
\begin{align*}
\left\| f_2 \right\|_{L^p(\mathbb{T} \times \mathbb{R})} & \lesssim
\delta^{\frac{1}{p}}  \lambda^{1-\frac{2}{p}} (\delta/ \lambda)^{-\frac{1}{4}+\frac{3}{2p}-\epsilon} ( \delta^{5/2} \lambda^{-3/2})^{\frac{1}{2} - \frac{1}{p}} \delta^{-1/2} \| f \|_{L^2} \\
& \lesssim \lambda^{\frac{1}{2} - \frac{2}{p}} \delta^{1/2}
\end{align*}

\bigskip

\noindent \underline{Optimality.} The optimality of the statement of the theorem is proved through two examples. The first one is an analog of the Knapp example: assume $\lambda \in \mathbb{N}$, and consider the function $g$ given by its Fourier transform
$$
\widehat{g}(k,\eta) = \mathbf{1}_{\lambda}(k) \chi \left( \frac{\eta}{\sqrt{\lambda \delta}} \right).
$$
Here, $\mathbf{1}_{\lambda}$ is the indicator function of $\{ \lambda \}$ and $\chi$ is a cutoff function with a sufficiently small support, so that $\operatorname{Supp} \widehat{g} \subset \mathcal{C}_{\lambda,\delta}$. In physical space,
$$
g(x,y) = \sqrt{\lambda \delta} e^{2\pi i\lambda x} \widehat{\chi}( \sqrt{\lambda \delta} y).
$$
It has $L^p$ norm $\sim (\lambda \delta)^{\frac{1}{2} - \frac{1}{2p}}$, so that
$$
\frac{ \| g \|_{L^p}}{\| g \|_{L^2}} \sim (\lambda \delta)^{\frac{1}{4} - \frac{1}{2p}}.
$$
We now consider the function $h$ given by its Fourier transform
$$
\widehat{h}(k,\eta) = \mathbf{1}_{\mathcal{C}_{\lambda,\delta}}(k,\eta) \mathbf{1}_{[0,\lambda/2]}(k);
$$
here, $\mathbf{1}_{\mathcal{C}_{\lambda,\delta}}$ is the indicator function of the annulus, and $\mathbf{1}_{[0,\lambda/2]}$ the indicator function of the interval. It is easy to check that
$|\operatorname{Supp} \widehat{h}| \sim \lambda \delta$, so that $\| h \|_{L^\infty} \sim \lambda \delta$ and $\| h \|_{L^2} \sim \sqrt{\lambda \delta}$, and finally
$$
\frac{ \| h \|_{L^\infty}}{\| h \|_{L^2}} \sim \sqrt{\lambda \delta}.
$$
By the Bernstein inequality,
$$
\frac{ \| h \|_{L^p}}{\| h \|_{L^2}} \gtrsim \lambda^{-2/p} \frac{ \| h \|_{L^\infty}}{\| h \|_{L^2}} \sim \lambda^{\frac{1}{2} - \frac{2}{p}} \delta^{1/2}.
$$
The examples $g$ and $h$ show that the statement of the theorem is optimal, up to subpolynomial losses.
\end{proof}

\section{Bounds for the Fourier support}

\begin{lem}[Bound on the size of Fourier support] \label{Fouriersupport}
\label{Fourier support}
With the notations of the proof of Theorem~\ref{mainthm},
\begin{itemize}
\item[(i)] The function $f_1$ is a function on $\mathbb{T} \times \mathbb{R}$. As such, its Fourier transform is supported on a union of lines, and has one-dimensional measure
$$
\displaystyle |  \operatorname{Supp} \widehat{f_1} | \lesssim \sqrt{\lambda \delta}
$$.
\item[(ii)] The function $\chi_\theta(D) F$ is a function on $\mathbb{R}^2$. Its Fourier transform is defined on $\mathbb{R}^2$, and has two-dimensional measure
$$
\displaystyle |  \operatorname{Supp} {\chi_\theta} \widehat{F} | \lesssim \delta^{5/2} \lambda^{-3/2}.
$$
\end{itemize}
\end{lem}

\begin{proof} $(i)$ Consider $f$ as in the proof of Theorem~\ref{mainthm}, namely with Fourier support in $\mathcal{C}_{\lambda,\delta}$. Since $(k,\eta)$ range in $\Z \times \R$ with \(k,\eta \geq 0\), the Fourier support of $f$ is contained in
$
\cup_{k \in \mathbb{Z}} \{ k \} \times E^\lambda_k,
$
where
\begin{equation*}
E_k^\la =
\begin{cases}
\emptyset
&
(k\geq k_+),
\\
\left(0,
\sqrt{k_+^2-k^2}
\right)
&(|k- \lambda| <\delta),
\\
\left(
\sqrt{(\lambda-\delta)^2-k^2},
\sqrt{k_+^2-k^2}
\right)
&(
0\leq k \leq \lambda-\delta).
\end{cases}
\end{equation*}
Recalling that \(\hat{f}_1\) is just \(\hat{f}\) restricted to \(|k-\lambda|\leq \frac{1}{\delta}\) , one can then add up these pieces to get the bound
\begin{align*}
| \operatorname{Supp}\widehat{ f_1 }|
&\lesssim  \sum_{\max\{0,\lambda-\frac{1}{\delta}\}\leq k< k_+} |E_k| \lesssim \sqrt{\lambda \delta} + \sum_{\substack{k \geq 0 \\ \frac{1}{\delta}\geq \lambda-k>\delta}}
\frac{{ \delta\lambda} }{\sqrt{\lambda (\lambda-x) }} 
 \intertext{and as \(y\mapsto 1/\sqrt{y}\) is decreasing this is}
 &\leq 2\sqrt{\lambda \delta} 
+ \int_\delta^{\min\{\lambda,\frac{1}{\delta}\}}
\frac{{ \delta\lambda} }{\sqrt{\lambda y}} \,dy
\leq
4\sqrt{\lambda\delta}.
\end{align*}
\bigskip

\noindent $(ii)$ Turning to $F$, it has Fourier support in
\begin{equation}
\label{ecureuil}
\bigcup_{\substack{k \in \mathbb{Z}\\ |k-\lambda|>\frac{1}{\delta} }}
\left[ \frac{k}{\lambda} - \frac{2 \delta}{\lambda},  \frac{k}{\lambda} + \frac{2\delta}{\lambda} \right] \times D^\lambda_k, \qquad D_k^\la = \left\{ H, \; 1 - \frac{\de}{\la} < \sqrt{\frac{k^2}{\lambda^2} + H^2} < 1 + \frac{\de}{\la} \right\}.
\end{equation}
Consider $\chi_\theta(D) F$, for a cap $\theta$ with dimensions  $\sim \frac{\delta}{\lambda} \times \sqrt{\frac{\delta}{\lambda}}$ adapted to the corona $\mathcal{C}_{1,3\delta/\lambda}$. Given such a cap, there is \(j\in\N\) with $2^j > \frac{1}{\delta}$ such that every point in the intersection of \(\theta\) with the set~\eqref{ecureuil} satisfies $|\lambda - k|\sim 2^j$. 

There are around $\sqrt{\delta} 2^{j/2}$ such values of \(k\) for which the vertical strip \(\left[ \frac{k}{\lambda} - \frac{2 \delta}{\lambda},  \frac{k}{\lambda} + \frac{2\delta}{\lambda} \right] \) intersects the cap $\theta$. For each such \(k\), the size of $D_k^\lambda$ is $\sim \delta \lambda^{-1/2} 2^{-j/2}$. Hence, adding up the contributions in \eqref{ecureuil},
\[
| \operatorname{Supp} \chi_\theta (D) F | \lesssim
\sqrt{\delta} 2^{j/2} \cdot \delta \lambda^{-1} \cdot \delta \lambda^{-1/2} 2^{-j/2} = \delta^{5/2} \lambda^{-3/2}.\qedhere
\]
\end{proof} 

\bibliographystyle{abbrv}
\bibliography{references}

\end{document}